\documentclass[11pt]{amsart} 
\usepackage{amssymb,amsfonts,graphicx,amsmath,amsthm,amstext,latexsym,amsfonts}
\usepackage{graphicx,epsfig,color}
\usepackage{bbm}
\usepackage{t1enc}
\usepackage{mathrsfs}
\usepackage{subfig}
\usepackage{hyperref}
\usepackage{a4wide}

\usepackage[foot]{amsaddr}

\usepackage{graphicx}        
\usepackage{tikz}
 \usetikzlibrary{arrows}
 \usetikzlibrary{calc}
\usepackage{pgfplots}
\usetikzlibrary{intersections}
\theoremstyle{plain}

\numberwithin{equation}{section}
\newtheorem{theorem}{Theorem}[section]
\newtheorem{proposition}[theorem]{Proposition}
\newtheorem{lemma}[theorem]{Lemma}

\newcommand{\LL}{\mathcal L}

\usepackage{color}
\definecolor{darkred}{rgb}{0.8,0,0}
\definecolor{darkblue}{rgb}{0,0,0.7}
\definecolor{darkgreen}{rgb}{0,0.4,0}

\newcommand{\KKK}{\color{black}}
\newcommand{\PPP}{\color{black}}

\newcommand{\eps}{\varepsilon}

\newcommand{\un}{{\rm 1\kern -2.5pt l}}

\def\u{\mathbf{u}}

\def\eps{\varepsilon}

\def\eps{\varepsilon}

\def\u{\mathbf{u}}

\renewcommand{\epsilon}{\varepsilon}
\newcommand{\beeq}{\begin{equation}}
\newcommand{\eneq}{\end{equation}}
\newcommand{\bear}{\begin{array}}
\newcommand{\enar}{\end{array}}
\newcommand{\bema}{\begin{displaymath}}
\newcommand{\enma}{\end{displaymath}}

\newcommand{\beea}{\begin{eqnarray}}
\newcommand{\enea}{\end{eqnarray}}

\title[]{\PPP A variational approach to nonlocal heat equations}
   \author[]{ Edoardo Mainini }
\date{}

 \address{Universit\`{a} degli Studi di Genova, Dipartimento di   Ingegneria Meccanica, Energetica, $\quad$ Gestionale e dei Trasporti (DIME),
  Via all'Opera Pia, 15 - 16145 Genova Italy.}
  \email{ edoardo.mainini@unige.it}
\begin{document}
 \maketitle
\begin{abstract} \PPP
We discuss a weighted variational integral approach for nonlocal linear diffusion models with forcing term, providing a selection principle for solutions of elliptic in time regularizations.
 \KKK
\end{abstract}
\begin{center}
\end{center}
\begin{flushleft}
  {\bf AMS Classification Numbers (2020):\,} 35S10, 49J45, 70H30
\end{flushleft}
\vskip0.5cm

\section{Introduction}

We consider the following initial value problem
\begin{equation}\label{mainproblem}
\left\{\begin{array}{ll}u'+\mathcal Lu=f\qquad&\mbox{in $(0,+\infty)\times\mathbb R^N$}\\
u(0)=u_0\qquad&\mbox{in $\mathbb R^N$}\end{array}\right.
\end{equation}
where $\mathcal L:\mathrm{Dom}(\mathcal L)\subset L^2(\mathbb R^N)\to L^2(\mathbb R^N)$ is a  self-adjoint Fourier multiplier with symbol $\hat L$, the prime $'$ stands for time derivative, 
and $f$ is a given space-time depending forcing term.\\

Besides the classical heat equation, \eqref{mainproblem} covers general nonlocal diffusion models like the fractional heat equation \cite{BSV,V,V2},
 obtained by letting
\begin{equation}\label{levy}
\mathcal L u(x):=p.v.\int_{\mathbb R^N}\kappa(x-y)(u(x)-u(y))\,dy
\end{equation}
with 
 $$\kappa(x)=c_{N,s}|x|^{-N-2s}, 
 \qquad 0<s<1,$$
 where $c_{N,s}>0$ is a suitable normalization constant. That is,  $\mathcal L=(-\Delta)^s$, the fractional power of the Laplacian operator, characterized in terms of Fourier transform as $$\widehat{(-\Delta)^sv}(\xi)=\hat L(\xi)\hat v(\xi)=|\xi|^{2s}\hat v(\xi),\qquad v\in\mathrm{Dom}(\mathcal L),$$ with $\mathrm{Dom}(\mathcal L)$ given by the fractional Sobolev space $H^{2s}(\mathbb R^N)$. \\

Among the main applications of nonlocal diffusion models, we might also consider $\mathcal L$ in the form \eqref{levy} with symmetric kernel $\kappa\in L^1(\mathbb R^N)$, 
leading to the zeroth order operator $\mathcal L u=Mu-\kappa\ast u$, $M=\int_{\mathbb R^N} u$. In this case,  $\hat L(\xi)=M-\hat \kappa(\xi)$ belongs to $L^\infty(\mathbb R^N)$, $\mathrm{Dom}(\mathcal L)=L^2(\mathbb R^N)$, and in general $\mathcal L$ is not a nonnegative operator. Nonlocal diffusion models of this form appear in biology, population dynamics and several more applications, see  \cite{BBCK,BFRW,BP,CCR,F} and references therein.\\

To the main equation \eqref{mainproblem} we
 associate the weighted variational integral 
\[
\mathcal F_\eps(u)=\int_{0}^{+\infty}\int_{\mathbb R^N}e^{-t/\eps}\left(\frac\eps2|u'|^2+\frac12 u\;\mathcal L u-fu\right)\,dx\,dt
\]
which is naturally related to the elliptic in time  regularization 
\begin{equation}\label{sec}
-\eps u''+ u'+\mathcal L u= f\qquad\mbox{in $(0,+\infty)\times\mathbb R^N$}.
\end{equation}
Indeed, the latter is the Euler-Lagrange equation of functional $\mathcal F_\eps$. Our objective is to show that, under a natural growth assumption on $f$ with respect to the time variable, and assuming that $\hat L$ is bounded from below, functional $\mathcal F_\eps$ admits a unique  minimizer $u_\eps$ among functions $u$ satisfying the initial condition $u(0,\cdot)=u_0(\cdot)$, 
and that $u_\eps$ converge a.e. in space-time to the solution of  \eqref{mainproblem} as $\eps\to0$. This is rigorously stated in Theorem \ref{maintheorem} below, which is the main result.\\

The exponentially weighted variational approach is typically used for hyperbolic problems. In such setting, it was introduced by De Giorgi \cite{DG,DG2}, who conjectured that a
solution to the initial value problem for the following nonlinear wave equation (with an integer $k>1$)
\begin{equation*}
\left\{\begin{array}{ll}u''-\Delta u+k u^{2k-1}=0\qquad&\mbox{in $(0,+\infty)\times\mathbb R^N$}\\
u(0)=u_0\qquad&\mbox{in $\mathbb R^N$}\\u'(0)=u_1\qquad&\mbox{in $\mathbb R^N$},
\end{array}\right.
\end{equation*}
with smooth compactly supported initial data,
is obtained as space-time a.e. limit as $\eps \to 0$ of minimizers of the weighted inertia energy (WIE) functional
\[
\mathcal I_\eps(u)=\int_{0}^{+\infty}\int_{\mathbb R^N}e^{-t/\eps}\left(\frac{\eps^2}2|u''|^2+\frac12|\nabla u|^2+\frac12 |u|^{2k}\right)\,dx\,dt
\]
among functions satisfying the same initial conditions.
This conjecture has been proved  in \cite{ST}, see also \cite{S} for the case of a finite time interval, and then extended to more general hyperbolic problems with dissipation in \cite{ABMS,ST2}, of course also including the classical linear wave equation.  The nontrivial extension to the nonhomogeneous case is treated in \cite{MP2}, see also \cite{TT1,TT2}.   This approach has also been often applied to parabolic problems, see \cite{AST,BDM,BDMS1, BDMS2, M}, see also  \cite{PT} where the authors treat a nonlocal problem like \eqref{mainproblem}. For a general overview we refer to the survey \cite{S2}. \\

The WIE approach can be used for treating ODEs as well, see \cite{LS,LS2,MP}, and here in
 this paper we are going to discuss such a variational approach   for the case of the linear first order equation \eqref{mainproblem}, which we rephrase  as a ODE in Fourier variables. This allows to obtain a new insight on this method and also on the role of the conditions on the forcing term.  Indeed, the unique minimizer for $\mathcal F_\eps$ among functions having initial value $u_0$ is shown to be the unique finite energy solution to \eqref{sec} having initial value $u_0$. Therefore, this method provides a selection principle for solutions of the elliptic in time regularization \eqref{sec}, by posing a condition on their behavior for large time ensuring finiteness of the energy. Such a condition depends on the forcing term, which is required to satisfy a fairly general growth assumption. 
Before stating and proving the main result in Section \ref{section3}, we devote Section \ref{section2} to the finite dimensional version, which is a basic example providing the main elements of the proof.
  
%
%
%

\section{The case of linear ODEs}\label{section2}
In this section we discuss a toy problem, i.e., the general ODE
\begin{equation}\label{ode}
\left\{\begin{array}{ll}y'+Ay=f(t)\qquad&\mbox{in $(0,+\infty)$}\\
y(0)=y_0\end{array}\right.
\end{equation}
where $A\in\mathbb R^{N\times N}$ is a symmetric matrix, $y_0\in\mathbb R^N$ and $f:(0,+\infty)\to\mathbb R^N$ is a forcing term that is assumed to be measurable and square Laplace transformable, i.e., we assume that
\begin{equation}\label{toyassumption}
\int_{0}^{+\infty}e^{-{t/\eps}}|f(t)|^2\,dt <+\infty\qquad \mbox{if $\eps>0$ is small enough}.
\end{equation}
The goal is to check that the unique solution to problem \eqref{ode} can be recovered as limit of minimizers  of functional
\[
\mathcal G_{\eps}(y):=\int_{0}^{+\infty}e^{-t/\eps}\left(\frac\eps2|y'|^2+\frac12\langle Ay,y\rangle-\langle f,y\rangle\right)\,dt,
\]
among functions $y$ such that $y(0)=y_0$,
as $\eps\to0$. The proof of this fact will serve as a reference also for the case of nonlocal heat equations. Indeed, we will check that this approach consists in a selection principle for solutions $y_\eps$ of the second order  approximation
\begin{equation}\label{viscous}
\left\{\begin{array}{ll}\eps y''=y'+Ay-f(t)\qquad&\mbox{in $(0,+\infty)$}\\
y(0)=y_0,\end{array}\right.
\end{equation}
ensuring convergence of $y_\eps$ to the solution of \eqref{ode}. On top of that, we shall check that there is a unique  solution to \eqref{viscous} having  finite $\mathcal G_{\eps}$ energy (for fixed $\eps$). Finiteness of the energy is therefore a necessary and sufficient condition for solutions to \eqref{viscous} to approximate the solution to \eqref{ode} as $\eps\to 0$.  

Let us introduce the functional setting. We let 
\[
\mathcal W_\eps:=\left\{y\in W^{1,1}_{loc}((0,+\infty);\mathbb R^N):\int_{0}^{+\infty}e^{-t/\eps}|y'(t)|^2\,dt<+\infty\right\}.
\]
A function in  $\mathcal W_\eps$ admits a continuous representative that is in $AC([0,T])$ for every $T>0$, and integration by parts along with Young inequality entails
\[\begin{aligned}
\int_0^{T}e^{-t/\eps}|y(t)|^2\,dt&=\left[-\eps e^{-t/\eps}|y(t)|^2\right]_0^T+2\eps\int_{0}^{T}e^{-t/\eps}\langle y(t),y'(t)\rangle\,dt\\
&\le \eps |y(0)|^2+\frac1{2}\int_0^Te^{-t/\eps}|y(t)|^2\,dt+2\eps^2\int_0^Te^{-t/\eps}|y'(t)|^2\,dt,
\end{aligned}\]  
therefore the following estimate holds for every $y\in \mathcal W_\eps$
\begin{equation}\label{fromy'toy}
\frac12\int_0^{+\infty}e^{-t/\eps}|y(t)|^2\,dt\le \eps |y(0)|^2+2\eps^2\int_0^{+\infty}e^{-t/\eps}|y'(t)|^2\,dt.
\end{equation}
In particular, $\mathcal W_\eps$ is the weighted Sobolev space of functions that are in $L^2_w((0,+\infty);\mathbb R^N)$ along with their distributional derivatives, where the weight $w$ is $e^{-t/\eps}\,dt$. In view of \eqref{fromy'toy} and \eqref{toyassumption}, functional $\mathcal G_\eps$ is well defined and finite on every $y\in\mathcal W_\eps$ if $\eps$ is small enough. Still  by using \eqref{toyassumption} and \eqref{fromy'toy}, it is not difficult the check that functional $\mathcal G_\eps$ is bounded from below over the the set $\{y\in\mathcal W_\eps:y(0)=y_0\}$ if $\eps$ is small enough. 

A critical point over $\mathcal W_\eps$ for functional $\mathcal G_\eps$ is naturally  defined as a function $y\in \mathcal W_\eps$ such that
\[\lim_{\delta\to0}\frac{\mathcal G_\eps(y+\delta\varphi)-\mathcal G_\eps(y)}{\delta}=
\int_0^{+\infty}e^{-t/\eps}\left(\eps\langle y'(t),\varphi'(t)\rangle+\langle Ay(t),\varphi(t)\rangle-\langle f(t),\varphi(t)\rangle\right)\,dt\]
vanishes for every $\varphi \in\mathcal W_\eps$. If $y\in \mathcal W_\eps$ is a critical point,  integrating by parts we get
\[
\int_0^{+\infty}e^{-t/\eps}\left\langle-\eps y''(t)+y'(t)+Ay(t)-f(t)\right\rangle\varphi(t)\,dt=0\qquad\mbox{for every $\varphi\in C^1_c((0,+\infty))$.}
\]  Therefore any critical point over $\mathcal W_\eps$ for functional $\mathcal G_\eps$
 is necessarily a solution to $\eps y''=y'+Ay-f$.
 In particular, any minimizer of $\mathcal G_\eps$ over the class $\{y\in\mathcal W_\eps: y(0)=y_0\}$ is necessarily a solution to $\eps y''=y'+Ay-f$ (noticing that perturbations with $C^1_c((0,+\infty))$ functions preserve the initial condition).
  However, we shall check that not every solution to the equation $\eps y''=y'+Ay-f$  has finite energy, i.e., not every solution belongs to $\mathcal W_\eps$. 
We are lead to search for a solution to \eqref{viscous} that belongs to $\mathcal W_\eps$, and we next check that it exists and is unique, thus being the unique minimizer of $\mathcal G_\eps$ over $\{y\in\mathcal W_\eps: y(0)=y_0\}$.

\begin{proposition}\label{propo} Assume \eqref{toyassumption}. Let $A\in\mathbb R^{N\times N}$ be symmetric. Let $y_0\in\mathbb R^N$. For every small enough $\eps>0$, there exists a unique minimizer $y_\eps$ for $\mathcal G_\eps$ in the class $\{y\in\mathcal W_\eps: y(0)=y_0\}$. As $\eps\to 0$, $y_\eps$ converge pointwise in $[0,+\infty)$ and uniformly on bounded intervals to
\[
y(t)=e^{-At}\left(y_0+\int_0^te^{As}f(s)\,ds\right),
\]
i.e., to the unique solution to \eqref{ode}.
\end{proposition}

\begin{proof}

Introducing the variable $z:[0,+\infty)\to\mathbb R^{2N}$ defined by $z(t):=(y(t),y'(t))$, \eqref{viscous} is equivalent to
\[z'=A_\eps z+f_\eps,\]
where  the $\mathbb R^{2N\times 2N}$  block matrix $A_\eps$ and the $\mathbb R^{2N}$ vector $f_\eps$ are defined by
\begin{equation}\label{Aepsfeps}A_\eps:=\begin{pmatrix}
 0 & I_N\\ \frac1\eps A & \frac1\eps I_N
 \end{pmatrix},\qquad f_\eps(t):=\begin{pmatrix}
 0 \\  -\frac1\eps f(t)
 \end{pmatrix}
\end{equation}
where $I_N$ denotes the $N\times N$ identity matrix.
Since
\[\det (A_\eps-\lambda I_{2N})=\det\left(\lambda(\lambda-\frac1\eps)I_N-\frac1\eps A\right)=\eps^{-N}\det(\lambda(\eps\lambda-1)I_N-A),\]
we see that $\lambda$ is an eigenvalue of $A_\eps$ if and only if $\lambda(\eps\lambda-1)$ is an eigenvalue of $A$. Therefore, letting $\mu^1\le\mu^2\le\ldots\le \mu^N$ the eigenvalues of $A$ and $4\eps<1/|\mu_1|$, the eigenvalues of $A_\eps$ are real and given by
\begin{equation}\label{eig}
\mu_\eps^i=:\frac{1-\sqrt{1+4\eps\mu^i}}{2\eps},\qquad\bar\mu^i_\eps:=\frac{1+\sqrt{1+4\eps\mu^i}}{2\eps},\qquad i=1,\ldots,N.
\end{equation}
It is not difficult to check that, for every $i=1,\ldots,N$, the eigenvectors respectively associated to $\mu_\eps^i$ and to $\bar\mu^i_\eps$ are the $\mathbb R^{2N}$ vectors
\[h_\eps^i:=\begin{pmatrix}
 h^i \\ \mu_\eps^i \,h^i
 \end{pmatrix},\qquad \bar h_\eps^i:=\begin{pmatrix}
 h^i \\  \bar\mu^i_\eps\, h^i
 \end{pmatrix}
\]
where $h^i\in\mathbb R^N$ is the normalized eigenvector of $A$ associated to the eigenvalue $\mu^i$, and we also define $P\in\mathbb R^{N\times N}$ as the orthogonal matrix whose columns are $h^1,\ldots, h^N$.
Let now $P_\eps\in\mathbb R^{2N\times 2N}$ a  matrix that diagonalizes $A_\eps$, i.e., we let $P_\eps$ the matrix whose columns are $h_\eps^1,\ldots h_\eps^N,\bar h_\eps^1,\ldots \bar h_\eps^N$. 
It can be written as the block matrix
\begin{equation}\label{block}P_\eps=\begin{pmatrix}
 P & P\\ Q_\eps & \bar Q_\eps 
 \end{pmatrix},
\end{equation}
where $Q_\eps$ is the matrix with columns $(\mu_\eps^1h^1,\ldots\mu_\eps^Nh^N)$ and $\bar Q_\eps$ is the matrix with columns $(\bar\mu_\eps^1h^1,\ldots, \bar\mu_\eps^Nh^N)$. Since $\mu_\eps^i\to-\mu^i$ and $\eps\bar\mu_\eps^i\to 1$ as $\eps\to 0$, we deduce that
\begin{equation}\label{matrixlimit}
\lim_{\eps\to 0}\eps Q_\eps=0,\qquad\lim_{\eps\to0}\eps\bar Q_\eps=P.
\end{equation}
The standard substitution $z=P_\eps X$ allows to rephrase  $z'=A_\eps z+f_\eps$ as the linear decoupled system
\begin{equation}\label{deca}
X'=D_\eps X+q_\eps,
\end{equation}
where $D_\eps:=P_\eps^{-1}A_\eps P_\eps$ is the diagonal matrix with eigenvalues $\mu_\eps^1,\ldots\mu_\eps^N,\bar\mu_\eps^1,\ldots,\bar\mu_\eps^N$ and where $q_\eps:=P_\eps^{-1}f_\eps$. 
We will make use of the following notation for solutions  $X_\eps$ to the system \eqref{deca}:  $x_\eps$ is the vector containing  the first $N$ components of $X_\eps$, while $\bar x_\eps$ includes the components of $X_\eps$ from $N+1$ to $2N$, so that $X_\eps=(x_\eps,\bar x_\eps)=(x_\eps^1,\ldots x_\eps^N,\bar x_\eps^1,\ldots\bar x_\eps^N)$, and similarly we let $q_\eps=(g_\eps,\bar g_\eps)=(g_\eps^1,\ldots g_\eps^N,\bar g_\eps^1,\ldots \bar g_\eps^N)$.  
In particular, the relation $q_\eps:=P_\eps^{-1}f_\eps$ and the definition of $f_\eps$ from \eqref{Aepsfeps} imply
\[
-\begin{pmatrix}
 0 \\ f
 \end{pmatrix}=\begin{pmatrix}
 \eps P & \eps P \\  \eps Q_\eps & \eps \bar Q_\eps
 \end{pmatrix}\begin{pmatrix}g_\eps\\\bar g_\eps\end{pmatrix}
\]
so that $\bar g_\eps=-g_\eps$ and taking \eqref{matrixlimit} into account we get
\begin{equation*}
\lim_{\eps\to 0} g_\eps(t)=P^{-1}f(t)\quad \mbox{ for a.e. $t>0$}.
\end{equation*}
The general solution $X_\eps$ to the decoupled system \eqref{deca}, since $q_\eps=(g_\eps,-g_\eps)$, is  given by
\begin{equation}\label{generalsolution}\begin{aligned}
 x_\eps^i(t)&=  
 e^{\mu_\eps^it}\left(x^i_\eps(0)+\displaystyle\int_0^te^{-\mu_\eps^is}g_\eps^i(s)\,ds\right) \qquad\mbox{ for $i=1,\ldots N$}\\
\bar x_\eps^i(t) &=
e^{\bar\mu_\eps^{i}t}\left(\bar x^i_\eps(0)-\displaystyle\int_0^te^{-\bar\mu_\eps^{i}s}g_\eps^{i}(s)\,ds\right) \qquad\mbox{ for $i=1,\ldots N$}
 \end{aligned}
\end{equation}

\smallskip
As  observed at the beginning of this section, if $y_\eps\in\mathcal W_\eps$ is a minimizer of $\mathcal G_\eps$ over $\{y\in\mathcal W_\eps: y(0)=y_0\}$  then $y_\eps$ solves \eqref{viscous}. We claim that there exists  one and only one solution $y_\eps$ to \eqref{viscous} satisfying $y_\eps\in\mathcal W_\eps$, thus trivially coinciding with the unique minimizer of $\mathcal G_\eps$ over $\{y\in\mathcal W_\eps: y(0)=y_0\}$.
Indeed, suppose $y_\eps$ is a solution to \eqref{viscous}. We let as before
$z_\eps=(y_\eps,y'_\eps)$  and $X_\eps=P_\eps^{-1}z_\eps$, so that $X_\eps$ is indeed a solution to \eqref{deca}, thus being one of the functions from  \eqref{generalsolution}. We have from \eqref{block} that \begin{equation}\label{xp}y_\eps=P(x_\eps+\bar x_\eps),\end{equation} therefore $|y_\eps|^2=|x_\eps+\bar x_\eps|^2$ since $P$ is orthogonal.
On the other hand, the condition $y_\eps\in \mathcal W_\eps$ implies
\[
\int_0^{+\infty}e^{-t/\eps}|y_\eps(t)|^2\,dt<+\infty
\]
in view of \eqref{fromy'toy}. Therefore, a necessary condition on 
  $(x_\eps,\bar x_\eps)$ in order to have $y_\eps\in\mathcal W_\eps$ is
\begin{equation}\label{xbarx}
\int_0^{+\infty}e^{-t/\eps}(x_\eps^i(t)^2+\bar x_\eps^i(t)^2)\,dt<+\infty \qquad \mbox{for every $i=1,\ldots,N$},
\end{equation}
where the condition involving $\bar x_\eps$, recalling \eqref{generalsolution} and \eqref{eig}, requires that for every $i=1,\ldots,N$
\begin{equation}\label{newsol}
\int_0^{+\infty}e^{\frac{\sqrt{1+4\eps\mu^i}}{\eps}t}\left|\bar x^i_\eps(0)-\int_0^te^{-\bar \mu_\eps^i s}g_\eps^i(s)\,ds\right|^2\,dt<+\infty.
\end{equation}
Since $q_\eps =P_\eps^{-1}f_\eps$, $q_\eps=(g_\eps,-g_\eps)$ and $f_\eps=-(0,f/\eps)$, and since $\bar\mu_\eps^i\ge1/(2\eps)$, condition \eqref{toyassumption} entails that the map $s\mapsto e^{-\bar\mu_\eps^i s}g_\eps^i(s)$ is in $L^1(0,+\infty)$ for every small enough $\eps$, therefore  we have the finite limit
\[
\lim_{t\to+\infty}\int_0^te^{-\bar \mu_\eps^i s}g_\eps^i(s)\,ds=\int_0^{+\infty} e^{-\bar \mu_\eps^i s}g_\eps^i(s)\,ds
\]
and thus
a necessary condition for the validity of \eqref{newsol} is
\begin{equation}\label{nec}
\bar x^i_\eps(0)=\int_0^{+\infty}e^{-\bar \mu_\eps^i s}g_\eps^i(s)\,ds\qquad\mbox{for every $i=1,\ldots N$,}
\end{equation}
and we next check that it is also a sufficient condition for obtaining \eqref{newsol}.
Indeed, under the validity of such a condition, for every $i=1,\ldots,N$ we have  by H\"older inequality
\[\begin{aligned}
\int_0^{+\infty}e^{-t/\eps}\bar x^i_\eps(t)^2\,dt&=\int_{0}^{+\infty}e^{\frac{\sqrt{1+4\eps\mu^i}}{\eps}t}\left|\int_t^{+\infty}e^{-\bar\mu_\eps^i s}g_\eps^i(s)\,ds\right|^2\,dt\\&\le \int_{0}^{+\infty}e^{\frac{\sqrt{1+4\eps\mu^i}}{\eps}t}\int_t^{+\infty}e^{-\bar\mu_\eps^is}\,ds\int_t^{+\infty}e^{-\bar\mu_\eps^is}|g_\eps^i(s)|^2\,ds\,dt\\&=\frac1{\bar\mu_\eps^i}\int_0^{+\infty}e^{\frac{\sqrt{1+4\eps\mu^i}}{\eps}t}e^{-\bar\mu_\eps^i t}\int_t^{+\infty}e^{-\bar\mu_\eps^is}|g_\eps^i(s)|^2\,ds\,dt,
\end{aligned}\]
hence by recalling the definition of $\bar \mu_\eps^i$ from \eqref{eig} and the property $\bar\mu_\eps^i\ge 1/(2\eps)$, and  since $|g_\eps^i(s)|\le |P_\eps^{-1}||f(s)|/\eps$ for a.e. $s>0$ (here the matrix norm is the operator norm), we deduce
\[\begin{aligned}
\int_0^{+\infty}e^{-t/\eps}\bar x^i_\eps(t)^2\,dt&\le \frac{1}{\bar\mu_\eps^i}\int_0^{+\infty}e^{-t/(2\eps)}\int_t^{+\infty}e^{-s/(2\eps)}|g_\eps^i(s)|^2\,ds\,dt\\&=\frac{2\eps}{\bar\mu_\eps^i}\int_0^{+\infty}e^{-s/(2\eps)}|g_\eps^i(s)|^2\,ds\le\frac{2|P_\eps^{-1}|}{\eps\bar\mu_\eps^i}\int_0^{+\infty}e^{-s/(2\eps)}|f(s)|^2\,ds<+\infty
\end{aligned}\]
for every small enough $\eps>0$,
thanks to \eqref{toyassumption}. On the other hand, the validity of \eqref{xbarx}, for the part involving $x_\eps$, is obtained from \eqref{eig} and from the elementary estimate $\mu_\eps^i\le -2(\mu^1\wedge0)$, since H\"older inequality yields
\[\begin{aligned}
&\int_0^{+\infty}e^{-t/\eps}x^i_\eps(t)^2\,dt\le\int_0^{+\infty}e^{-t/\eps}e^{2\mu_\eps^it}\left|x^i_\eps(0)+\int_0^te^{-\mu_\eps^is}g_\eps^i(s)\,ds\right|^2\,dt\\
&\qquad\le \frac{2\eps x^i(0)^2}{\sqrt{1+4\eps\mu^i}}+2\int_0^{+\infty}te^{-t/\eps}\int_0^te^{2\mu_\eps^i(t-s)}|g_\eps^i(s)|^2\,ds\,dt\\
&\qquad\le \frac{2\eps x^i(0)^2}{\sqrt{1+4\eps\mu^i}}+2\int_0^{+\infty}te^{-t/(2\eps)}e^{-4(\mu^1\wedge0)t}\int_0^te^{-s/(2\eps)}|g_\eps^i(s)|^2\,ds\,dt\\
&\qquad\le \frac{2\eps x^i(0)^2}{\sqrt{1+4\eps\mu^i}}+\frac{8|P_\eps^{-1}|}{(1+4(\mu_1\wedge0)\eps)^2}\int_0^{+\infty}e^{-s/(2\eps)}|f(s)|^2\,ds\,dt<+\infty
\end{aligned}\]
for every small enough $\eps>0$, where finiteness of the integral is again due to \eqref{toyassumption}. We have shown that \eqref{nec} is a necessary and sufficient condition for having $y_\eps\in\mathcal W_\eps$, and it identifies a unique solution of \eqref{viscous}, because once $\bar x_\eps(0)$ is assigned the relation \eqref{xp} at $t=0$ identifies uniquely the vector $x_\eps(0)$ as the solution to the linear system \begin{equation}\label{Pxy}Px_\eps(0)=y_0-P\bar x_\eps(0).\end{equation} The claim is proven, and thus for every fixed small enough $\eps>0$ there exists a unique minimizer $y_\eps\in\mathcal W_\eps$ for $\mathcal G_\eps$ over  $\{y\in\mathcal W_\eps: y(0)=y_0\}$, and it is characterized as the unique solution of \eqref{viscous} having finite $\mathcal G_\eps$ energy. Its expression is  \eqref{xp}, with
\begin{equation}\label{allsolutions}\begin{aligned}
 x_\eps^i(t)&=  
 e^{\mu_\eps^it}\left(x_\eps^i(0)+\displaystyle\int_0^te^{-\mu_\eps^is}g_\eps^i(s)\,ds\right) \qquad\mbox{ for $i=1,\ldots N,$}\\
\bar x_\eps^i(t) &=
e^{\bar\mu_\eps^{i}t}\displaystyle\int_t^{+\infty}e^{-\bar\mu_\eps^{i}s}g_\eps^{i}(s)\,ds \qquad\mbox{ for $i=1,\ldots N$}
 \end{aligned}
\end{equation}
 and with $x_\eps(0)=P^{-1}(y_0-P\bar x_\eps(0))$ and $\bar x_\eps(0)$ given by \eqref{nec}.
 
 \smallskip
 
 Let us eventually check that such solution $y_\eps$ converges uniformly on compact intervals to the unique solution of \eqref{ode}.
We first recall that due to $\mathbb R^{2N}\ni q_\eps=P_\eps^{-1}f_\eps$, where $f_\eps=-(0,f/\eps)$ and $q_\eps=(g_\eps,-g_\eps)$, by \eqref{block} we have $g_\eps=(\eps\bar Q_\eps-\eps Q_\eps)^{-1}f$, and since $\eps Q_\eps\to 0$, $\eps\bar Q_\eps\to P$, and $|P|=|P^{-1}|=1$, we have, for a.e. $t>0$,
$|g_\eps(t)|\le|(\eps\bar Q_\eps-\eps Q_\eps)^{-1}||f(t)|\le 2|f(t)|$  (having chosen $\eps$ small enough).  
We have as a consequence, similarly as before by H\"older inequality and since $\bar\mu_\eps^i\ge 1/(2\eps)$
\begin{equation}\label{buona}
\begin{aligned}
&\sup_{t\in[0,T]}\left|\int_t^{+\infty}e^{-\bar\mu_\eps^i(s-t)}g_\eps^i(s)\,ds\right|^2\le 4\sup_{t\in[0,T]}\int_t^{+\infty}e^{(t-s)/(2\eps)}\,ds\int_t^{+\infty}e^{(t-s)/(2\eps)}|f(s)|^2\,ds\\&
\qquad\le 8\eps \sup_{t\in[0,T]}\int_t^{2T}e^{(t-s)/(2\eps)}|f(s)|^2\,ds+8\eps \sup_{t\in[0,T]}\int_{2T}^{+\infty}e^{(t-s)/(2\eps)}|f(s)|^2\,ds\\
&\qquad\le 8\eps\sup_{t\in[0,2T]}\int_t^{2T}e^{(t-s)/(2\eps)}|f(s)|^2\,ds+8\eps\int_{2T}^{+\infty}e^{-s/(4\eps)}|f(s)|^2\,ds\\
\end{aligned}
\end{equation}
where the second term the right hand side vanishes as $\eps\to0$ by dominated convergence due to \eqref{toyassumption}, and the first term in the right hand side vanishes as well by invoking Lemma \ref{tech} below. Therefore $\bar x_\eps(t)\to 0$ as $\eps\to0$, uniformly on bounded intervals. In particular $\bar x_\eps(0)\to 0$ as $\eps\to0$ so that indeed from the relation \eqref{Pxy} we see that $x_\eps(0)\to P^{-1}y_0$. 
As a consequence, it is readily seen that $e^{\mu_\eps^it}x_\eps^i(0)\to e^{-\mu^i}t(P^{-1}y_0)^i$ as $\eps\to0$ for every $i=1,\ldots, N$, uniformly on bounded time intervals. On the other hand by H\"older inequality
\begin{equation}\label{buona2}\begin{aligned}
&\sup_{t\in[0,T]}\left|\int_0^t e^{\mu_\eps^i(t-s)}|g_\eps(s)-P^{-1}f(s)|\,ds\right|^2\\&\qquad\le Te^{-4(\mu^1\wedge0)T} \int_0^T|g_\eps(s)-P^{-1}f(s)|^2\,ds\\&\qquad\le Te^{-4(\mu^1\wedge0)T} |(\eps\bar Q_\eps-\eps Q_\eps)^{-1}-P^{-1}|\int_0^T|f(s)|^2\,ds 
\end{aligned}\end{equation}
where the right hand side vanishes as $\eps\to0$ since \eqref{toyassumption} implies $f\in L^2(0,T)$ for every $T>0$ and since $\eps Q_\eps\to0$, $\eps\bar Q_\eps\to P$. Since the right hand sides of both \eqref{buona} and \eqref{buona2} are vanishing as $\eps\to0$ from \eqref{allsolutions} and \eqref{Pxy} we see that  uniformly on bounded time intervals $\bar x_\eps\to0$ and,  for every $i=1,\ldots,N$, $x_\eps^i$ converge uniformly on bounded time intervals to
\[
x^i(t):=e^{-\mu^it}\left(x_0^i+\int_0^te^{\mu^is}(P^{-1}f(s))^i\,ds\right), \qquad x_0^i:=(P^{-1}y_0)^i.
\]
Thus the  limit $y$ of $y_\eps$ is given by $Px$, thanks to \eqref{xp}, where
 $x$ is  the solution to $x'=Dx+P^{-1}f$ with initial datum $x(0)=x_0=P^{-1}y_0$, where $D\in \mathbb R^{N\times N}$ is the diagonal matrix with eigenvalues $-\mu^1,\ldots,-\mu^N$. Since  the columns of the orthogonal matrix $P$ are the eigenvectors of $-A$, there holds $D=-P^{-1}AP$.
Going back to the variable $y$, we see that
\[
y'=Px'=P(Dx+P^{-1}f)=PDP^{-1}y+f=-Ay+f,\qquad y(0)=Px(0)=y_0,
\]
 so $y$ is indeed the solution to \eqref{ode}.
\end{proof}

The following simple lemma has been used through the proof of Proposition \ref{propo}.
\begin{lemma}\label{tech} Let $T\in(0,+\infty)$ and $g\in L^1(0,T)$. For $\eps>0$ and $t\in[0,T]$, let
\[
G_\eps(t):=e^{t/\eps}\int_t^Te^{-s/\eps}|g(s)|\,ds.
\]
Then $G_\eps\to 0$ in $C^0([0,T])$ as $\eps\to0$.
\end{lemma}
\begin{proof} The result is trivial if $|g|\equiv 1$. By possibly switching from $|g|$ to $|g|+1$, it is therefore not restrictive to assume that $|g|\ge1$.
Let $\eps\in(0,e^{-4})$, let $p_\eps:=\sqrt{-\log\eps}$, let $p_\eps':=p_\eps/(p_\eps-1)$ and let moreover $g_\eps:=(p_\eps-1)\wedge |g|$. Clearly $g_\eps\ge1$ and $g_\eps\to |g|$ in $L^1(0,T)$ as $\eps\to0$. Therefore, after having noticed  that
\[\begin{aligned}
e^{t/\eps}\int_t^Te^{-s/\eps}|g(s)|\,ds&=e^{t/\eps}\int_t^Te^{-s/\eps}g_\eps(s)\,ds+e^{t/\eps}\int_t^Te^{-s/\eps}(|g(s)|-g_\eps(s))\,ds\\&\le e^{t/\eps}\int_t^Te^{-s/\eps}g_\eps(s)\,ds+\int_0^T(|g(s)|-g_\eps(s))\,ds,
\end{aligned}\] we are left to prove that
\begin{equation}\label{tseps}\sup_{t\in[0,T]}e^{t/\eps}\int_t^Te^{-s/\eps}g_\eps(s)\,ds\to0\end{equation}
as $\eps\to0$. We have by H\"older inequality
\[e^{t/\eps}\int_t^Te^{-s/\eps}g_\eps(s)\,ds\le e^{t/\eps}\left(\int_t^T e^{-sp_\eps/\eps}\,ds\right)^{1/p_\eps}\|g_\eps\|_{L^{p_\eps'}(0,T)}
\le \left(\frac\eps{p_\eps}\right)^{1/p_\eps}\,\|g_\eps\|_{L^{p_\eps'}(0,T)}
\]
and we claim that the right hand side goes to $0$ as $\eps\to0$, thus proving \eqref{tseps}. Since $(1/p_\eps)^{1/p_\eps}\to 1$ and $p_\eps'\to 1$, and since
\[
\frac{\eps^{p_\eps'/p_\eps}}T\int_0^Tg_\eps(s)^{p_\eps'}\,ds\le \eps^{p_\eps'/p_\eps}(p_\eps-1)^{p_\eps'}=e^{-p_\eps p_\eps'}(p_\eps-1)^{p_\eps'}=\,e^{-p_\eps^2/(p_\eps-1)}\,(p_\eps-1)^{p_\eps/(p_\eps-1)}\to0,
\] 
the claim follows, so that \eqref{tseps} holds. The proof is concluded.
\end{proof}

\section{Nonlocal heat equations}\label{section3}

In this section, the operator $\mathcal L:\mathrm{Dom}(\mathcal L)\subset L^2(\mathbb R^N)\to L^2(\mathbb R^N)$  appearing in equation \eqref{mainproblem} is assumed to be a self-adjoint Fourier multiplier, with  $\mathcal S(\mathbb R^N)\subset\mathrm{Dom}(\mathcal L)$. We also assume  that its symbol $\hat L$  is bounded from below over $\mathbb R^N$, i.e.
\begin{equation}\label{boundfrombelow}
\mbox{there exists $K\le 0$ such that }\quad \hat L(\xi)\ge K\;\mbox{ for a.e. $\xi\in\mathbb R^N$}
\end{equation}
 We further assume that the forcing term $f\in L^1_{loc}((0,+\infty)\times\mathbb R^N)$ satisfies the following property: the map $t\mapsto\|f(t,\cdot)\|_{L^2(\mathbb R^N)}^2$ is Laplace transformable, i.e., 
\begin{equation}\label{transformability}\mbox{for every small enough $\eps>0$ there holds
}\;\;\;
\int_0^{+\infty}e^{-t/\eps}\|f(t,\cdot)\|_{L^2(\mathbb R^N)}^2\,dt<+\infty.
\end{equation}

We define the functional setting for the main result. Let 
\[
V_{\mathcal L}:=\left\{u\in L^2(\mathbb R^N): \int_{\mathbb R^N}|\hat L(\xi)||\hat u(\xi)|^2\,d\xi<+\infty\right\}.
\]
The space $V_{\mathcal L}$ is the domain of the  quadratic form $(u,v)\mapsto\frac12\int_{\mathbb R^N}u\,\mathcal Lv\,dx$  associated with $\mathcal L$ and appearing in functional $\mathcal F_\eps$. It is normed by
\[
\|v\|_{V_{\mathcal L}}:=\left(\int_{\mathbb R^N}(1+|\hat L(\xi)|)|v(\xi)|^2\,d\xi\right)^{1/2}.
\]
We further define the weighted Sobolev space 
$$
H_\eps:=\left\{u\in W^{1,1}_{loc}((0,+\infty);L^2(\mathbb R^N)): \int_{0}^{+\infty}e^{-t/\eps}\|u'(t,\cdot)\|^2_{L^2(\mathbb R^N)}<+\infty\right\}.
$$
We notice that for every $u\in H_\eps$  there holds by Young inequality
\begin{equation*}\begin{aligned}
&\int_{0}^{T}e^{-t/\eps}\|u(t,\cdot)\|^2_{L^2(\mathbb R^N)}\,dt\le \left[-\eps e^{-t/\eps}\|u(t,\cdot)\|^2_{L^2(\mathbb R^N)}\right]_0^T+\eps\int_0^Te^{-t/\eps}\frac{d}{dt}\|u(t,\cdot)\|^2_{L^2(\mathbb R^N)}\,dt\\
&\qquad\le\eps\|u(0,\cdot)\|^2_{L^2(\mathbb R^N)}+\eps\int_0^{T} e^{-t/\eps}\left(\frac1{2\eps}\|u(t,\cdot)\|^2_{L^2(\mathbb R^N)}+2\eps\|u'(t,\cdot)\|^2_{L^2(\mathbb R^N)}\right)\,dt
\end{aligned}
\end{equation*}
so that
\begin{equation}\label{fromu'tou}\begin{aligned}
&\frac12\int_{0}^{+\infty}e^{-t/\eps}\|u(t,\cdot)\|^2_{L^2(\mathbb R^N)}\,dt\le 
\eps\|u(0,\cdot)\|^2_{L^2(\mathbb R^N)}+2\eps^2\int_0^{+\infty} e^{-t/\eps}\|u'(t,\cdot)\|^2_{L^2(\mathbb R^N)}\,dt.
\end{aligned}
\end{equation}
We finally define the linear space
\[
H^{\mathcal L}_\eps:=\left\{u \in H_\eps: \int_0^{+\infty}e^{-t/\eps}\int_{\mathbb R^N}|\hat L(\xi)||\hat u(t,\xi)|^2\,d\xi\,dt<+\infty\right\}.
\]
Note that here and in the following the hat $\hat{}$ symbol denotes Fourier transform in the space variable.
For every $u\in H^{\mathcal L}_\eps$ we have by Plancherel theorem
\[
\mathcal F_\eps(u)=\mathcal J_\eps(\hat u):=\int_{0}^{+\infty}e^{-t/\eps}\int_{\mathbb R^N}\left(\frac\eps2 |\hat u(t,\xi)|^2+\frac12\,\hat L(\xi)|\hat u(t,\xi)|^2-\hat f(t,\xi)\overline{\hat u(t,\xi)}\right)\,d\xi\,dt.
\]
The Gateaux derivative of $\mathcal F$ at $u\in H^{\mathcal L}_\eps$ is defined as the linear operator
\begin{equation}\label{derivative}\begin{aligned}
d\mathcal F_\eps(u)[\varphi]:&=\lim_{\delta\to0}\frac{\mathcal F_\eps(u+\delta\varphi)-\mathcal F_\eps(u)}{\delta}\\&=\int_{0}^{+\infty}e^{-t/\eps}\int_{\mathbb R^N}\left(\eps u'\varphi'+\varphi\mathcal L u-f\varphi\right)\,dx\,dt,\qquad\varphi \in H^{\mathcal L}_\eps,
\end{aligned}
\end{equation}
where we used the fact that $\mathcal L$ is self-adjoint. Similarly as above, making use of Plancherel theorem, the Gateaux derivative of $\mathcal F$ is related to the Gateaux derivative of $\mathcal J$ by
\[
d\mathcal J_\eps(\hat u)[\hat \varphi]=d \mathcal F_\eps(u)[\varphi]\qquad\mbox{for every $u\in H_\eps^\LL$ and every $\varphi\in H_\eps^\LL.$}
\]
A critical point of functional $\mathcal F_\eps$ over $H^{\mathcal L}_\eps$ is naturally defined 
as a function
$u\in H^{\mathcal L}_\eps$ such that   $d \mathcal F(u)[\varphi]=0\;\;\mbox{ for every }\;\;\varphi\in H^{\mathcal L}_\eps.$
Before stating the main result, let us preliminarily obtain a basic estimate for the solution to \eqref{mainproblem}.
\begin{proposition} Let $u_0\in V_{\mathcal L}$. 
There exists a unique $u \in C^0([0,+\infty);L^2(\mathbb R^N))$ solving \eqref{mainproblem} and it is given by the inverse Fourier transform of
\begin{equation}\label{original}
\hat u(t,\xi)=e^{-\hat L(\xi)t}\left(\hat u_0(\xi)+\int_0^te^{\hat L(\xi)s}\hat f(s,\xi)\,ds\right).
\end{equation}
Moreover, $u\in C^0([0,+\infty);V_{\mathcal L})$ and for every $T>0$
\[\sup_{t\in[0,T]}\|u(t,\cdot)\|^2_{V_{\mathcal L}}\le 2(1-K)(T+1)e^{-2KT}\left(\|u_0\|^2_{V_{\mathcal L}}+\int_0^T\|f(s,\cdot)\|^2_{L^2(\mathbb R^N)}\right),
\]
where $K\le0$ is the constant from \eqref{boundfrombelow}.
\end{proposition}
\begin{proof} The Fourier transform in space of the equation in \eqref{mainproblem} is $\hat u'+\hat L\,\hat u=\hat f$, with the initial condition $\hat u(0,\xi)=\hat u_0(\xi)$. Therefore, $\hat u$ takes the form \eqref{original}. Moreover, taking advantage of \eqref{boundfrombelow} we have
the following estimate for a.e. $\xi\in\mathbb R^N$
\[
\sup_{t\in [0,T]}(1+|\hat L(\xi)|)\int_0^te^{\hat L(\xi)(s-t)}\,ds\le (1-K)(T+1)e^{-KT}.
\]
Indeed, if $\hat L(\xi)\le 0$ this is immediately obtained by using $|\hat L(\xi)|=-\hat L(\xi)\le -K$, else if $\hat L(\xi)>0$ this is obtained by noticing that 
\[
\sup_{t\in[0,T]}(1+|\hat L(\xi)|)\int_0^t e^{-\hat L(\xi)(t-s)}\,ds=\sup_{t\in [0,T]}\left(\frac{1-e^{-\hat L(\xi)t}}{\hat L(\xi)}+(1-e^{-\hat L(\xi)t})\right)\le T+1 
\]
since $1-e^{-x}\le x$ for every $x>0$.
Therefore by H\"older inequality
\begin{equation}\label{firstdominant}\begin{aligned}
&\sup_{t\in[0,T]}(1+|\hat L(\xi)|)\left|\int_0^t e^{\hat L(\xi)(s-t)}\hat f(s,\xi)\right|^2\,ds\\&\qquad\le\sup_{t\in[0,T]} (1+|\hat L(\xi)|)\int_{0}^te^{\hat L(\xi)(s-t)}\,ds\int_0^te^{\hat L(\xi)(s-t)}|\hat f(s,\xi)|^2\,ds\,\\
&\qquad \le (1-K)(T+1)e^{-2KT}\int_0^T|\hat f(s,\xi)|^2\,ds\qquad\mbox{for a.e. $\xi\in\mathbb R^N$}.
\end{aligned}
\end{equation}
By using \eqref{boundfrombelow} again and by Fubini theorem we deduce
\[\begin{aligned}
\sup_{t\in[0,T]}\|\hat u(t,\cdot)\|^2_{V_{\mathcal L}}&\le2\sup_{t\in[0,T]}\int_{\mathbb R^N} (1+|\hat L(\xi)|)e^{-2\hat L(\xi)t}|\hat u_0(\xi)|^2\,d\xi\\&\qquad+2\sup_{t\in[0,T]}\int_{\mathbb R^N}(1+|\hat L(\xi)|)\left|\int_0^te^{\hat L(\xi)(s-t)}\hat f(s,\xi)\,ds\right|^2\,d\xi\\
&\le 2e^{-2KT}\|u_0\|^2_{V_{\mathcal L}}+2(1-K)(T+1)e^{-2KT}\int_0^T\|\hat f(s,\cdot)\|^2_{L^2(\mathbb R^N)}\,ds
\end{aligned}\]
thus concluding the proof.
\end{proof}

We now state the main result, proving that the solution to problem \eqref{mainproblem}, with $u_0\in V_{\mathcal L}$, is the limit of minimizers of functional $\mathcal F_\eps$ over $\{u\in H_\eps^{\mathcal L}:u(0,x)=u_0(x)\mbox{ for a.e. $x\in\mathbb R^N$}\}$ as $\eps\to 0$.

\begin{theorem}\label{maintheorem}
 Assume \eqref{boundfrombelow} and \eqref{transformability}. Let $u_0\in V_{\mathcal L}$. For every small enough $\eps>0$, there exists a unique function $u_\eps\in H_\eps^\LL$ such that $u_\eps$ is a minimizer of $\mathcal F_\eps$  over $\{u\in H_\eps^{\mathcal L}:u(0,x)=u_0(x)\mbox{ for a.e. $x\in\mathbb R^N$}\}$. Moreover, letting $u$ denote the unique solution to problem \eqref{mainproblem}, then as $\eps\to 0$ we have $u_\eps\to u$ a.e. in space-time, and for every $T>0$
\[
u_\eps\to u\;\;\mbox{in $C^0([0,T];V_{\mathcal L})$}.
\]
\end{theorem}

\begin{proof}
We preliminarily observe that functional $\mathcal F_\eps$, for every fixed small enough $\eps$, is bounded from below over the set $\{u\in H^{\mathcal L}_\eps: u(0,x)=u_0(x) \mbox{ for a.e. $x\in\mathbb R^N$}\}$. This is a direct consequence of \eqref{boundfrombelow}, \eqref{transformability} and \eqref{fromu'tou}, since Young inequality entails
\[\begin{aligned}
&\left|\int_0^{+\infty}e^{-t/\eps} \int_{\mathbb R^N}f(t,x)u(t,x)\,dx\,dt\right|\\
&\qquad\le 4\eps\int_0^{+\infty}e^{-t/\eps}\|f(t,\cdot)\|^2_2\,dt+\frac18\|u_0\|^2_2+\frac\eps4\int_0^{+\infty}e^{-t/\eps}\|u'(t,\cdot)\|_2^2\,dt,
\end{aligned}\]
with the shorthand $\|\cdot\|_2$ for $\|\cdot\|_{L^2(\mathbb R^N)}$, here and in the rest of the proof.
Moreover there holds
\[\begin{aligned}
&\frac12\int_0^{+\infty}e^{-t/\eps}\int_{\mathbb R^N}u\,\mathcal Lu\,dx\,dt\ge \frac K2\int_0^{+\infty}e^{-t/\eps}\int_{\mathbb R^N}|u(t,x)|^2\,dx\,dt\\
&\qquad\ge K\eps\|u_0\|^2_2+2K\eps^2\int_0^{+\infty}e^{-t/\eps}\|u'(t,\cdot)\|_2^2\,dt.
\end{aligned}\]
As a consequence,  for every small enough $\eps>0$ we have the following bound from below
\[\begin{aligned}
\mathcal F(u_\eps)&= \frac\eps2\int_0^{+\infty}e^{-t/\eps}\|u'(t,\cdot)\|^2_2\,dt-\int_0^{+\infty}e^{-t/\eps} \int_{\mathbb R^N}f(t,x)u(t,x)\,dx\,dt\\&
\qquad+\frac12\int_0^{+\infty}\int_{\mathbb R^N}u\,\mathcal Lu\,dx\,dt\ge -4\eps\int_0^{+\infty}e^{-t/\eps}\|f(t,\cdot)\|^2_2\,dt+\left(K\eps-\frac18\right)\|u_0\|^2_2.
\end{aligned}\]

We also preliminarily observe that if $u_\eps\in H^\LL_\eps$ is a critical point of $\mathcal F_\eps$, then for every  $\varphi\in C^{\infty}_c((0,+\infty)\times\mathbb R^N)$ by \eqref{derivative} we have
\[
\int_{0}^{+\infty}e^{-t/\eps}\int_{\mathbb R^N}\left(\eps u_\eps'\varphi'+\varphi\mathcal L u_\eps-f\varphi\right)\,dx\,dt=0.
\]
Integration by parts in time entails
\[
\int_{0}^{+\infty}e^{-t/\eps}\int_{\mathbb R^N}\left(-\eps u_\eps''+u_\eps'+\mathcal L u_\eps-f\right)\varphi\,dx\,dt=0.
\]
This shows that, in Fourier variables, for a.e. $\xi\in\mathbb R^N$ the function $\hat u_\eps(\cdot,\xi)$ satisfies the ODE
\begin{equation}\label{sec2}
-\eps\hat u''+\hat u'+\hat L(\xi)\hat u=\hat f,\quad t>0.
\end{equation}
In particular, a minimizer of $\mathcal F_\eps$  over $\{u\in H_\eps^{\mathcal L}:u(0,x)=u_0(x)\mbox{ for a.e. $x\in\mathbb R^N$}\}$ has Fourier transform that necessarily satisfies \eqref{sec2}.\\

{\textbf{Step 1: uniqueness.} }
Suppose  that $u_\eps\in H^\LL_\eps$ is a minimizer of $\mathcal F_\eps$ over $\{u\in H_\eps^{\mathcal L}:u(0,x)=u_0(x)\mbox{ for a.e. $x\in\mathbb R^N$}\}$. Thus it satisfies \eqref{sec2}. 
The characteristic polynomial of the operator in the right hand side of \eqref{sec2} is $-\eps\lambda^2+\lambda+\hat L(\xi)$, therefore for every small enough $\eps$ (independently on $\xi$, since $\hat L$ is bounded from below) its  roots are real and given by
\[
\lambda_\eps(\xi):=\frac{1-Z_\eps(\xi)}{2\eps},\qquad\mu_\eps(\xi):=\frac{1+Z_\eps(\xi)}{2\eps},\quad\qquad\mbox{where}\quad Z_\eps(\xi):=\sqrt{1+4\eps\hat L(\xi)}.
\]
Moreover, the general solution of \eqref{sec2} is given, for every fixed $\xi$, by
\begin{equation}\label{twolines}\begin{aligned}
\hat v_\eps(t,\xi)&=e^{\lambda_\eps(\xi)t}\left(c_1(\xi)+\frac1{Z_\eps(\xi)}\int_0^te^{-\lambda_\eps(\xi)s}\hat f(s,\xi)\,ds\right)\\&\qquad+e^{\mu_\eps(\xi)t}\left(c_2(\xi)-\frac1{Z_\eps(\xi)}\int_0^te^{-\mu_\eps(\xi)s}\hat f(s,\xi)\,ds\right),
\end{aligned}\end{equation}
and by recalling  \eqref{fromu'tou},  $v_\eps$ belonging to $H_\eps$ requires by Fubini and Plancherel theorem that
\begin{equation}\label{condizione}
\int_0^{+\infty}e^{-t/\eps}|\hat v_\eps(t,\xi)|^2\,dt<+\infty\quad\mbox{for a.e. $\xi\in \mathbb R^N$.}
\end{equation}
We notice that by \eqref{twolines} we have
\[
e^{-t/(2\eps)}\hat v_\eps(t,\xi)=a_\eps(t,\xi)+b_\eps(t,\xi),
\]
where
\[
a_\eps(t,\xi):=e^{Z_\eps(\xi)t/(2\eps)}\left(c_2(\xi)-\frac1{Z_\eps(\xi)}\int_0^te^{-\mu_\eps(\xi)s}\hat f(s,\xi)\,ds\right),
\]
\[
b_\eps(t,\xi):=e^{-Z_\eps(\xi)t/(2\eps)}\left(c_1(\xi)+\frac1{Z_\eps(\xi)}\int_0^te^{-\lambda_\eps(\xi)s}\hat f(s,\xi)\,ds\right)
\]
As an immediate consequence of \eqref{transformability} we have, for every small enough $\eps $,
\[\limsup_{t\to+\infty}|b_\eps(t,\xi)|\le\int_0^{+\infty}e^{-s/(2\eps)}|\hat f(s,\xi)|\,ds<+\infty\quad\mbox{ for a.e. $\xi\in\mathbb R^N$}\]
and
\[
\lim_{t\to+\infty}\int_0^te^{-\mu_\eps(\xi)s}\,|\hat f(s,\xi)|\,ds=\int_0^{+\infty}e^{-\mu_\eps(\xi)s}\,|\hat f(s,\xi)|\,ds<+\infty\quad\mbox{ for a.e. $\xi\in\mathbb R^N$}.
\]
Therefore, a necessary condition for the validity of \eqref{condizione} is 
\[
c_2(\xi)=\frac1{Z_\eps(\xi)}\int_0^{+\infty}e^{-\mu_\eps(\xi)s}\hat f(s,\xi)\,ds\quad\mbox{ for a.e. $\xi\in\mathbb R^N$}.
\]
Inserting this in \eqref{twolines}, we deduce that being $u_\eps$ a minimizer satisfying the initial condition $\hat u(0,\xi)=\hat u_0(\xi)$ for a.e. $\xi\in\mathbb R^N$, then its Fourier transform in space  must take the form 
\begin{equation}\label{ueps}\begin{aligned}
\hat u_\eps(t,\xi)&=e^{\lambda_\eps(\xi)t}\left(\hat u_0(\xi)-\frac1{Z_\eps(\xi)}\int_0^{+\infty}e^{-\mu_\eps(\xi)s}\hat f(s,\xi)\,ds\right)\\&\quad+\frac{e^{\lambda_\eps(\xi)t}}{Z_\eps(\xi)}\int_0^te^{-\lambda_\eps(\xi)s}\hat f(s,\xi)\,ds+\frac{e^{\mu_\eps(\xi)t}}{Z_\eps(\xi)}\int_t^{+\infty}e^{-\mu_\eps(\xi)s}\hat f(s,\xi)\,ds.
\end{aligned}\end{equation}

{\textbf{Step 2: existence}}.
Let us check that the inverse Fourier transform of the function from \eqref{ueps} belongs indeed to $H_\eps^\LL$, thus showing it is the unique minimizer of $\mathcal F_\eps$ over $\{u\in H_\eps^{\mathcal L}:u(0,x)=u_0(x)\mbox{ for a.e. $x\in\mathbb R^N$}\}$, and in fact the only solution to \eqref{sec2} with initial datum $u_0$ and with finite $\mathcal F_\eps$ energy.
We have
\begin{equation}\label{ueps'}\begin{aligned}
\hat u'_\eps(t,\xi)&=\lambda_\eps(\xi)\,e^{\lambda_\eps(\xi)t}\left(\hat u_0(\xi)-\frac1{Z_\eps(\xi)}\int_0
^{+\infty}e^{-\mu_\eps(\xi)s}\hat f(s,\xi)\,ds\right)\\&\quad+\frac{\lambda_\eps(\xi)}{Z_\eps(\xi)}\,e^{\lambda_\eps(\xi)}\int_0^te^{-\lambda_\eps(\xi)s}\hat f(s,\xi)\,ds+\frac{\mu_\eps(\xi)}{Z_\eps(\xi)}\,e^{\mu_\eps(\xi)t}\int_t^{+\infty}e^{-\mu_\eps(\xi)s}\hat f(s,\xi)\,ds.
\end{aligned}
\end{equation}
We remark that \eqref{boundfrombelow} implies that for every small enough $\eps>0$ (such that $4\eps|K|<1/2$) the following estimates hold for a.e. $\xi\in\mathbb R^N$
\begin{equation}\label{ae}
Z_\eps(\xi)\ge\frac1{\sqrt2},\qquad\frac{|\hat L(\xi)|}{Z_\eps(\xi)^2}\le\frac1{4\eps},\qquad\left|\frac{\lambda_\eps(\xi)}{Z_\eps(\xi)}\right|\le \frac{\mu_\eps(\xi)}{Z_\eps(\xi)}\le \frac{1+\sqrt2}{2\eps},\qquad\mu_\eps(\xi)\ge\frac1{2\eps}\end{equation}
as well as
\begin{equation}\label{38}\lambda_\eps(\xi)\le -2K,\qquad |\lambda_\eps(\xi)|\le \sqrt{\frac{|\hat L(\xi)|}{\eps}},\qquad-\lambda_\eps(\xi)\le \hat L(\xi),\end{equation}
 obtained by using that $\sqrt{1+x}\ge 1+x$ for $x\in (-1,0)$, that $|1-\sqrt{1+y}|\le \sqrt{|y|}$ for $y\in(-1,+\infty)$, and that $\sqrt{1+z}-1\le z/2 $ for $z\in(-1,+\infty)$, respectively.
The estimates \eqref{ae} along with
 H\"older inequality and Fubini theorem entail
\begin{equation}\label{a1}\begin{aligned}&\int_0^{+\infty}\int_{\mathbb R^N}e^{-t/\eps}\left|\frac{\lambda_\eps(\xi)}{Z_\eps(\xi)}\right|^2\,e^{2\lambda_\eps(\xi)t}\left|\int_0^{+\infty}e^{-\mu_\eps(\xi)s}\hat f(s,\xi)\,ds\right|^2\,d\xi\,dt\\
&\qquad \le\int_0^{+\infty}\int_{\mathbb R^N}e^{-t/\eps}\left|\frac{\lambda_\eps(\xi)}{Z_\eps(\xi)}\right|^2\frac1{\mu_\eps(\xi)}\,e^{2\lambda_\eps(\xi)t}\left(\int_0^{+\infty}e^{-\mu_\eps(\xi)s}|\hat f(s,\xi)|^2\,ds\right)\,d\xi\,dt\\
&\qquad\le\frac{(1+\sqrt2)^2}{2\eps}\int_0^{+\infty}e^{-t/(\sqrt2\eps)}\,dt\int_0^{+\infty}e^{-s/(2\eps)}\|\hat f(s,\cdot)\|_2^2\,ds\\&\qquad\le \frac{\sqrt2}2(1+\sqrt2)^2\int_0^{+\infty}e^{-s/(2\eps)}\|\hat f(s,\cdot)\|_2^2\,ds<+\infty
\end{aligned}\end{equation}
for every small enough $\eps>0$, where finiteness is due to \eqref{transformability}, and similarly
\begin{equation}\label{b1}\begin{aligned}&\int_0^{+\infty}\int_{\mathbb R^N}e^{-t/\eps}\frac{|\hat L(\xi)|}{Z_\eps(\xi)^2}\,e^{2\lambda_\eps(\xi)t}\left|\int_0^{+\infty}e^{-\mu_\eps(\xi)s}\hat f(s,\xi)\,ds\right|^2\,d\xi\,dt\\
&\qquad \le\frac12\int_0^{+\infty}e^{-t/\eps}\,e^{2\lambda_\eps(\xi)t}\,dt \int_0^{+\infty}e^{-s/(2\eps)}\|\hat f(s,\cdot)\|^2_2\,ds\\
&\qquad\le\frac{\sqrt2}{2}\,\eps \int_0^{+\infty}e^{-s/(2\eps)}\|\hat f(s,\cdot)\|^2_2\,ds<+\infty.
\end{aligned}\end{equation}
The same arguments using \eqref{ae}-\eqref{38} also yield
\begin{equation}\label{a2}
\begin{aligned}
&\int_{0}^{+\infty}\int_{\mathbb R^N}e^{-t/\eps}\left|\frac{\lambda_\eps(\xi)}{Z_\eps(\xi)}\right|^2\,e^{2\lambda_\eps(\xi)t}\left|\int_0^te^{-\lambda_\eps(\xi)s}\hat f(s,\xi)\,ds\right|^2\,d\xi\,dt\\
&\qquad\le\left(\frac{1+\sqrt2}{2\eps}\right)^2\int_0^{+\infty}\int_{\mathbb R^N}te^{-t/\eps}\int_0^t e^{2\lambda_\eps(t-s)}|\hat f(s,\xi)|^2\,ds\,d\xi\,dt\\
&\qquad\le \left(\frac{1+\sqrt2}{2\eps}\right)^2\int_0^{+\infty}te^{-t/\eps}\int_0^te^{-4K(t-s)}\|\hat f(s,\cdot)\|_2^2\,ds\,dt\\
&\qquad \le \left(\frac{1+\sqrt2}{2\eps}\right)^2\int_0^{+\infty}te^{-4Kt}e^{-t/(2\eps)}\,dt\int_0^{+\infty}e^{-s/(2\eps)}\|\hat f(s,\cdot)\|_2^2\,ds\\
&\qquad =\left(\frac{1+\sqrt2}{1+8\eps K}\right)^2\int_0^{+\infty}e^{-s/(2\eps)}\|\hat f(s,\cdot)\|_2^2\,ds<+\infty
\end{aligned}
\end{equation}
and similarly
\begin{equation}\label{b2}
\begin{aligned}
&\int_{0}^{+\infty}\int_{\mathbb R^N}e^{-t/\eps}\frac{|\hat L(\xi)|}{Z_\eps(\xi)^2}\,e^{2\lambda_\eps(\xi)t}\left|\int_0^te^{-\lambda_\eps(\xi)s}\hat f(s,\xi)\,ds\right|^2\,d\xi\,dt\\
&\qquad\le\frac{\eps}{(1+8\eps K)^2}\int_0^{+\infty}e^{-s/(2\eps)}\|\hat f(s,\cdot)\|_2^2\,ds<+\infty.
\end{aligned}
\end{equation}
Furthermore, by the same reasoning, 
\begin{equation}\label{a3}
\begin{aligned}
&\int_{0}^{+\infty}\int_{\mathbb R^N}e^{-t/\eps}\left|\frac{\mu_\eps(\xi)}{Z_\eps(\xi)}\right|^2\,e^{2\mu_\eps(\xi)t}\left|\int_t^{+\infty}e^{-\mu_\eps(\xi)s}\hat f(s,\xi)\,ds\right|^2\,d\xi\,dt\\
&\quad\le \int_{0}^{+\infty}\int_{\mathbb R^N}e^{-t/\eps}\left|\frac{\mu_\eps(\xi)}{Z_\eps(\xi)}\right|^2\,e^{2\mu_\eps(\xi)t}\int_t^{+\infty}e^{-\mu_\eps(\xi)s}\,ds\int_t^{+\infty}e^{-\mu_\eps(\xi)s}|\hat f(s,\xi)|^2\,ds\,d\xi\,dt\\
&\quad\le \frac{(1+\sqrt2)^2}{2\eps}\int_0^{+\infty}\int_{\mathbb R^N}e^{-t/\eps}\int_t^{+\infty}e^{-\mu_\eps(s-t)}|\hat f(s,\xi)|^2\,ds\,d\xi\,dt\\
&\quad \le \frac{(1+\sqrt2)^2}{2\eps}\int_0^{+\infty}e^{-t/(2\eps)}\,dt\int_0^{+\infty}e^{-s/(2\eps)}\|\hat f(s,\cdot)\|_2^2\,ds\\&\quad=(1+\sqrt2)^2\int_0^{+\infty}e^{-s/(2\eps)}\|\hat f(s,\cdot)\|_2^2\,ds<+\infty
\end{aligned}
\end{equation}
and
\begin{equation}\label{b3}
\begin{aligned}
&\int_{0}^{+\infty}\int_{\mathbb R^N}e^{-t/\eps}\frac{|\hat L(\xi)|}{Z_\eps(\xi)^2}\,e^{2\mu_\eps(\xi)t}\left|\int_t^{+\infty}e^{-\mu_\eps(\xi)s}\hat f(s,\xi)\,ds\right|^2\,d\xi\,dt\\
&\qquad\le \eps \int_0^{+\infty}e^{-s/(2\eps)}\|\hat f(s,\cdot)\|_2^2\,ds<+\infty.
\end{aligned}
\end{equation}
Eventually, again with \eqref{ae}-\eqref{38} we have
\begin{equation}\label{a4}\begin{aligned}
&\int_0^{+\infty}\int_{\mathbb R^N}e^{-t/\eps}|\lambda_\eps(\xi)|^2\,e^{2\lambda_\eps(\xi)t}|\hat u_0(\xi)|^2\,d\xi\,dt\\&\qquad\le\frac1\eps\int_0^{+\infty}\int_{\mathbb R^N}e^{-t/\eps}\,e^{2\lambda_\eps(\xi)t}|\hat L(\xi)||\hat u_0(\xi)|^2\,d\xi\,dt\le \sqrt2 \int_{\mathbb R^N}|\hat L(\xi)||\hat u_0(\xi)|^2\,d\xi<+\infty\end{aligned}
\end{equation}
and
\begin{equation}\label{b4}
\int_0^{+\infty}\int_{\mathbb R^N}e^{-t/\eps}|\hat L(\xi)|\,e^{2\lambda_\eps(\xi)t}|\hat u_0(\xi)|^2\,d\xi\,dt\le\sqrt2 \eps\int_{\mathbb R^N}|\hat L(\xi)||\hat u_0(\xi)|^2\,d\xi<+\infty
\end{equation}
since $u_0\in V_{\mathcal L}$.
By recalling the expression of $\hat u'_\eps$ from \eqref{ueps'}, and by taking into account \eqref{a1}, \eqref{a2}, \eqref{a3}, \eqref{a4} we deduce  that
\[
\int_0^{+\infty}e^{-t/\eps}\|\hat u'_\eps(t,\cdot)\|^2_{L^2(\mathbb R^N)}\,dt<+\infty
\]
for every small enough $\eps>0$. Similarly, from \eqref{ueps} and taking into account \eqref{b1}, \eqref{b2}, \eqref{b3}, \eqref{b4} we infer
\[
\int_0^{+\infty}e^{-t/\eps}\int_{\mathbb R^N}|\hat L(\xi)||\hat u_\eps(t,\xi)|^2\,d\xi\,dt<+\infty
\]
for every small enough $\eps>0$. By Plancherel theorem, we have therefore $u_\eps\in H_\eps^{\mathcal L}$ for every small enough $\eps>0$.\\

{\textbf{Step 3: convergence.}}
We prove that $u_\eps\to u$ in $C^0([0,T];V_{\mathcal L}(\mathbb R^N))$ for every $T>0$, i.e.,
\begin{equation}\label{mainconvergence}
\lim_{\eps\to 0}\sup_{t\in[0,T]}\left\|u_\eps(t,\cdot)-u(t,\cdot)\right\|_{V_\mathcal L}=0.
\end{equation}
 The main estimates that follow are similar to the ones of the previous step.
We have by \eqref{ae}-\eqref{38} and H\"older inequality
\[\begin{aligned}
&\int_{\mathbb R^N}\left|\frac{e^{\lambda_\eps(\xi)t}}{Z_\eps(\xi)}\int_0^{+\infty}e^{-\mu_\eps(\xi)s}\hat f(s,\xi)\,ds\right|^2\,d\xi\\
&\qquad\le 2e^{-4Kt}\int_{\mathbb R^N}\frac1{\mu_\eps(\xi)}\int_0^{+\infty}e^{-\mu_\eps(\xi)s}|\hat f(s,\xi)\,ds\,d\xi\le 4\eps e^{-4Kt}\int_0^{+\infty}e^{-s/(2\eps)}\|\hat f(s,\cdot)\|_2^2\,ds,
\end{aligned}\]
therefore
\begin{equation}\label{c1}\begin{aligned}
&\lim_{\eps\to0}\sup_{t\in[0,T]}\int_{\mathbb R^N}\left|\frac{e^{\lambda_\eps(\xi)t}}{Z_\eps(\xi)}\int_0^{+\infty}e^{-\mu_\eps(\xi)s}\hat f(s,\xi)\,ds\right|^2\,d\xi\\&\qquad\le \lim_{\eps\to0}4\eps e^{-4KT}\int_0^{+\infty}e^{-s/(2\eps)}\|\hat f(s,\cdot)\|_2^2\,ds=0
\end{aligned}\end{equation}
where the integral in the right hand side is indeed vanishing as $\eps \to 0$ thanks to the dominated convergence theorem and to the assumption \eqref{transformability}. Similarly, by also using the second estimate in \eqref{ae} we have
\[
\begin{aligned}
&\int_{\mathbb R^N}|\hat L(\xi)|\left|\frac{e^{\lambda_\eps(\xi)t}}{Z_\eps(\xi)}\int_0^{+\infty}e^{-\mu_\eps(\xi)s}\hat f(s,\xi)\,ds\right|^2\,d\xi\le\frac12 e^{-4Kt}\int_0^{+\infty}e^{-s/(2\eps)}\|\hat f(s,\cdot)\|_2^2\,ds
\end{aligned}
\]
so that 
\begin{equation}\label{c2}\begin{aligned}
\lim_{\eps\to0}\sup_{t\in[0,T]}\int_{\mathbb R^N}|\hat L(\xi)|\left|\frac{e^{\lambda_\eps(\xi)t}}{Z_\eps(\xi)}\int_0^{+\infty}e^{-\mu_\eps(\xi)s}\hat f(s,\xi)\,ds\right|^2\,d\xi=0.
\end{aligned}\end{equation}
Moreover, \eqref{ae}-\eqref{38} allow to deduce with analogous reasoning  that
\[\begin{aligned}
\int_{\mathbb R^N}\left|\frac{e^{\mu_\eps(\xi)t}}{Z_\eps(\xi)}\int_t^{+\infty}e^{-\mu_\eps(\xi)s}\hat f(s,\xi)\,ds\right|^2\,d\xi
\le4\eps\int_t^{+\infty}e^{-(s-t)/(2\eps)}\|\hat f(s,\cdot)\|_2^2\,ds,
\end{aligned}\]
and that
\[
\int_{\mathbb R^N}|\hat L(\xi)|\left|\frac{e^{\mu_\eps(\xi)t}}{Z_\eps(\xi)}\int_t^{+\infty}e^{-\mu_\eps(\xi)s}\hat f(s,\xi)\,ds\right|^2\,d\xi
\le\frac12\int_t^{+\infty}e^{-(s-t)/(2\eps)}\|\hat f(s,\cdot)\|_2^2\,ds.
\]
On the other hand,
\[\begin{aligned}
&\sup_{t\in[0,T]}\int_t^{+\infty}e^{-(s-t)/(2\eps)}\|\hat f(s,\cdot)\|_2^2\,ds\\&\qquad\le
\sup_{t\in[0,T]}\int_t^{2T}e^{-(s-t)/(2\eps)}\|\hat f(s,\cdot)\|_2^2\,ds+\sup_{t\in[0,T]}\int_{2T}^{+\infty}e^{-(s-t)/(2\eps)}\|\hat f(s,\cdot)\|_2^2\,ds\\
&\qquad \le\sup_{t\in[0,2T]}\int_t^{2T}e^{-(s-t)/(2\eps)}\|\hat f(s,\cdot)\|_2^2\,ds+\int_{2T}^{+\infty}e^{-s/(4\eps)}\|\hat f(s,\cdot)\|_2^2\,ds
\end{aligned}\]
where the first term in right hand side vanishes as $\eps\to0$ due to Lemma \ref{tech}, since \eqref{transformability} implies that the function $s\mapsto \|\hat f(s,\cdot)\|_2^2$ is in $L^1(0,2T)$, while the second term in the right hand side vanishes by assumption \eqref{transformability} and by dominated convergence.  We conclude that
\begin{equation}\label{c3}
\lim_{\eps\to 0}\sup_ {t\in[0,T]} \int_{\mathbb R^N}(1+|\hat L(\xi)|)\left|\frac{e^{\mu_\eps(\xi)t}}{Z_\eps(\xi)}\int_t^{+\infty}e^{-\mu_\eps(\xi)s}\hat f(s,\xi)\,ds\right|^2\,d\xi=0
\end{equation}

We also notice that thanks to \eqref{38} and \eqref{boundfrombelow}
\[\begin{aligned}
&\sup_{t\in[0,T]}\left|e^{\lambda_\eps(\xi)t}-e^{-\hat L(\xi)t}\right|^2\le T^2\,e^{-4KT}|\lambda_\eps(\xi)-\hat L(\xi)|^2
\end{aligned}\]
so that $$\sup_{t\in[0,T]}\left|e^{\lambda_\eps(\xi)t}-e^{-\hat L(\xi)t}\right|^2\to 0$$ as $\eps\to 0$ for a.e. $\xi\in\mathbb R^N$,
and also
\[
\sup_{t\in[0,T]}\left|e^{\lambda_\eps(\xi)t}-e^{-\hat L(\xi)t}\right|^2\le 2\sup_{t\in[0,T]}\left(e^{2\lambda_\eps(\xi)t}+e^{-2\hat L(\xi)t}\right)\le 4e^{-4KT}
\]
for a.e. $\xi\in\mathbb R^N$. Therefore by dominated convergence theorem and since $u_0\in V_{\mathcal L}$ we deduce
\begin{equation}\label{c4}
\lim_{\eps\to 0}\sup_{t\in[0,T]}\int_{\mathbb R^N}(1+|\hat L(\xi)|)\left|e^{\lambda_\eps(\xi)t}-e^{-\hat L(\xi)t}\right|^2\,|\hat u_0(\xi)|^2\,d\xi=0
\end{equation}

We finally claim that as $\eps\to 0$
\begin{equation}\label{claim}
\int_{\mathbb R^N}(1+|\hat L(\xi)|)\sup_{t\in[0,T]}\left|\frac{e^{\lambda_\eps(\xi)t}}{Z_\eps(\xi)}\int_0^te^{-\lambda_\eps(\xi)s}\hat f(s,\xi)\,ds-\int_0^te^{-\hat L(\xi)(t-s)}\hat f(s,\xi)\,ds\right|^2\,d\xi\to0.
\end{equation}
We will prove the claim by dominated convergence.
By \eqref{boundfrombelow}, \eqref{ae} and \eqref{38} we directly have that if $\hat L(\xi)\le 0$, then
\[
\sup_{t\in[0,T]}\frac{1+|\hat L(\xi)|}{Z_\eps(\xi)^2}\int_0^te^{2\lambda_\eps(\xi)(t-s)}\,ds\le 2(1-K)Te^{-4KT},\]
whereas for $\hat L(\xi)>0$ (so that $\lambda_\eps(\xi)<0$) there holds
\[\begin{aligned}
&\sup_{t\in[0,T]}\frac{1+|\hat L(\xi)|}{Z_\eps(\xi)^2}\int_0^te^{2\lambda_\eps(\xi)(t-s)}\,ds\le \sup_{t\in[0,T]}\frac{1+|\hat L(\xi)|}{Z_\eps(\xi)^2}\frac{e^{2\lambda_\eps(\xi)t}-1}{2\lambda_\eps(\xi)}\\
&\qquad=\sup_{t\in[0,T]} \frac{\eps\hat L(\xi)(1-e^{2\lambda_\eps(\xi)t})}{(1+4\eps\hat L(\xi))(\sqrt{1+4\eps\hat L(\xi)}-1)}+\frac{e^{2\lambda_\eps(\xi)t}-1}{2\lambda_\eps(\xi)Z_\eps(\xi)^2}\le \frac12+2T,
\end{aligned}\]
having used the elementary inequalities $\frac{2x}{(1+4x)(\sqrt{1+4x}-1)}\le 1$ for every $x>0$ and $\frac{1-e^{-y}}{y}\le 1$ for every $y>0$. Combining the last two estimates we see that
\[
\sup_{t\in[0,T]}\frac{1+|\hat L(\xi)|}{Z_\eps(\xi)^2}\int_0^te^{2\lambda_\eps(\xi)(t-s)}\le 2(1-K)(T+1)e^{-4KT}\qquad\mbox{ for a.e. $\xi\in\mathbb R^N$},
\]
thus H\"older inequality entails
\begin{equation}\label{seconddominating}\begin{aligned}
&\sup_{t\in[0,T]}\frac{1+|\hat L(\xi)|}{Z_\eps(\xi)^2}\left|\int_0^te^{\lambda_\eps(\xi)(t-s)}\hat f(s,\xi)\right|^2\\&\qquad\le \sup_{t\in[0,T]}\frac{1+|\hat L(\xi)|}{Z_\eps(\xi)^2}\int_0^te^{2\lambda_\eps(\xi)(t-s)}\,ds\int_{0}^t|\hat f(s,\xi)|^2\,ds\\
&\qquad\le 2(1-K)(T+1)e^{-4KT}\int_0^T|\hat f(s,\xi)|^2\,ds\qquad\mbox{ for a.e. $\xi\in\mathbb R^N$}.
\end{aligned}\end{equation}
By \eqref{firstdominant} and \eqref{seconddominating}, we find that
$
h(\xi):=6(1-K)(T+1)e^{-4KT}\int_0^T|\hat f(s,\xi)|^2\,ds
$
is a  dominating function for the integrand in \eqref{claim}, noticing that indeed $h\in L^1(\mathbb R^N)$ as a consequence of assumption \ref{transformability}. Let us next prove the pointwise a.e. convergence to $0$ for the integrand in \eqref{claim}. Indeed, we have by \eqref{ae}-\eqref{38} that for a.e. $\xi\in\mathbb R^N$
\begin{equation}\label{firstae}\begin{aligned}
&\sup_{t\in[0,T]}\left|\frac{e^{\lambda_\eps(\xi)t}}{Z_\eps(\xi)}\int_0^te^{-\lambda_\eps(\xi)s}\hat f(s,\xi)\,ds-\frac{e^{\lambda_\eps(\xi)t}}{Z_\eps(\xi)}\int_0^te^{\hat L(\xi)s}\hat f(s,\xi)\,ds\right|\\&\qquad=\sup_{t\in[0,T]}\frac{e^{2\lambda_\eps(\xi)t}}{Z_\eps(\xi)^2}\int_0^t(e^{\hat L(\xi)s}-e^{-\lambda_\eps(\xi)s})|\hat f(s,\xi)|\,ds\\&\qquad\le 2e^{-4KT}\int_0^T(e^{\hat L(\xi)s}-e^{-\lambda_\eps(\xi)s})|\hat f(s,\xi)|\,ds.
\end{aligned}\end{equation}
For a.e. $\xi\in\mathbb R^N$, the above right hand side vanishes as $\eps\to 0$, since we have $-\lambda_\eps(\xi)\to \hat L(\xi) $ and  $(e^{\hat L(\xi)s}-e^{-\lambda_\eps(\xi)s})|\hat f(s,\xi)|\le e^{|\hat L(\xi)|T}|\hat f(s,\xi)|$ for every $s\in[0,T]$, where $s\mapsto f(s,\xi)$ is in $L^1((0,T))$ thanks to H\"older inequality and \eqref{transformability}. On the other hand, for a.e. $\xi\in\mathbb R^N$ we have
\begin{equation}\label{secondae}\begin{aligned}
&\sup_{t\in[0,T]}\left|\frac{e^{\lambda_\eps(\xi)t}}{Z_\eps(\xi)}\int_0^te^{\hat L(\xi)s}\hat f(s,\xi)\,ds-{e^{-\hat L(\xi)t}}\int_0^te^{\hat L(\xi)s}\hat f(s,\xi)\,ds\right|\\&\qquad\le \left(e^{|\hat L(\xi)|T}\int_0^T|\hat f(s,\xi)|\,ds\right)\sup_{t\in[0,T]}\left|\frac{e^{\lambda_\eps(\xi)t}}{Z_\eps(\xi)}-e^{-\hat L(\xi)t}\right|\end{aligned}\end{equation}
and we notice that by \eqref{boundfrombelow}, \eqref{ae}, \eqref{38}
\[\begin{aligned}
\sup_{t\in[0,T]}\left|\frac{e^{\lambda_\eps(\xi)t}}{Z_\eps(\xi)}-e^{-\hat L(\xi)t}\right|&\le \frac1{Z_\eps(\xi)}\sup_{t\in[0,T]}|e^{\lambda_\eps(\xi)t}-e^{-\hat L(\xi)t}|+\sup_{t\in[0,T]}\left|\frac{e^{-\hat L(\xi)}t}{Z_\eps(\xi)}-e^{-\hat L(\xi)t}\right|\\&\le \frac1{Z_\eps(\xi)}e^{-2KT}\sup_{t\in[0,T]}|\lambda_\eps(\xi)t+\hat L(\xi)t|+e^{-KT}\left|\frac1{Z_\eps(\xi)}-1\right|,
\end{aligned}\]
and the right hand side vanishes for a.e. $\xi\in\mathbb R^N$ as $\eps\to 0$, since $\lambda_\eps(\xi)\to-\hat L(\xi)$ and $Z_\eps(\xi)\to 1$.
Having shown that the right hand sides of \eqref{firstae} and \eqref{secondae} are both vanishing for a.e. $\xi\in\mathbb R^N$ as $\eps\to 0$, we infer that the same holds for the integrand in \eqref{claim} so that \eqref{claim} does hold and the claim is proved.

The combination of \eqref{c1}, \eqref{c2}, \eqref{c3}, \eqref{c4} and \eqref{claim} allows to obtain \eqref{mainconvergence}.
Since the right hand sides of of \eqref{firstae} and \eqref{secondae} are both vanishing for a.e. $\xi\in\mathbb R^N$ as $\eps\to 0$, we also deduce that 
\[
\frac{e^{\lambda_\eps(\xi)t}}{Z_\eps(\xi)}\int_0^te^{-\lambda_\eps(\xi)s}\hat f(s,\xi)\,ds\to\int_0^te^{-\hat L(\xi)(t-s)}\hat f(s,\xi)\,ds
\]
for a.e. $(t,\xi)\in (0,+\infty)\times\mathbb R^N$  as $\eps \to 0$, and similarly by taking into account the previous estimates we see that $u_\eps\to u$ a.e. in space-time as $\eps\to 0$.
\end{proof}

We conclude with a simple remark. We verify that existence of a minimizer of $\mathcal F_\eps$ over $\{u\in H_\eps^{\mathcal L}:u(0,x)=u_0(x)\mbox{ for a.e. $x\in\mathbb R^N$}\}$ can be obtained via direct method under a positivity assumption on the operator $\mathcal L$. In order to do this, let $\hat L\ge 0$. If $(u_n)_{n\in\mathbb N}\subseteq \{u\in H^{\mathcal L}_\eps: u(0,x)=u_0(x) \mbox{ for a.e. $x\in\mathbb R^N$}\}$ is a minimizing sequence for $\mathcal F_\eps$, then by \eqref{fromu'tou} we have that up to subsequences $u_n$ is weakly converging to $u$  in $L^2_{w}((0,+\infty)\times\mathbb R^N)$, where the weight $w$ is $e^{-t/\eps}\,dt\,dx$, and $u_n'$ is weakly converging to $u'$ in     $L^2_{w}((0,+\infty)\times\mathbb R^N)$. The weak lower semicontinuity of the $L^2$ norm, the continuity  of the linear operator $u\mapsto \int_0^{+\infty}e^{-t/\eps}\int_{\mathbb R^N}f(x,t)u(x,t)\,dx\,dt$, following from \eqref{transformability}, and the weak lower semicontinuity of the nonnegative quadratic form $u\mapsto \int_{0}^{+\infty}e^{-t/\eps}\int_{\mathbb R^N}u\,\mathcal Lu\,dx\,dt$ entail
\[
\mathcal F_\eps(u)\le\liminf_{n\to+\infty}\mathcal F_\eps(u_n).
\]
Moreover, the uniform bounds of $(u_n)$ and $(u_n')$ in $L^2((0,T);L^2(\mathbb R^N))$ imply the a.e. identities $u(0,\cdot)=u_n(0,\cdot)=u_0(\cdot)$, for every $n$,
thus showing the minimality of $u$. By Theorem \ref{maintheorem}, the obtained minimizer $u$ is  unique and it  is the inverse Fourier transform of the function defined in \eqref{ueps}.

\bigskip\smallskip

\textbf{Acknowledgements.} 
The author acknowledges support from the MUR-PRIN  project  No.  202244A7YL. He is member of the GNAMPA group of the Istituto Nazionale di Alta Matematica (INDAM).

%


\begin{thebibliography}{99}
\bibitem{ABMS} G. Akagi, V.  B\"ogelein, A. Marveggio, U. Stefanelli,  {\textit{Weighted inertia-dissipation-energy approach to doubly nonlinear wave equations}}, J. Funct. Anal. 289(8) (2025), 111067.

\bibitem{AST} A. Audrito, E. Serra,  P. Tilli, {\textit  {A minimization procedure to the existence of segregated
solutions to parabolic reaction-diffusion systems}}, Comm. Partial Differential Equations,
46(12) (2021), 2268--2287.

\bibitem{BBCK} M. T. Barlow, R. F. Bass, Z.-Q. Chen and M. Kassmann, {\textit{Non-local Dirichlet forms and symmetric jump processes}}, Trans. Amer. Math. Soc. 361 (2009), 1963--1999.

\bibitem{BFRW} P. Bates, P. Fife, X. Ren, X. Wang, {\textit{Travelling waves in a convolution model for phase transitions}}, Arch. Rat. Mech. Anal. 138 (1997), 105--136.


\bibitem{BDM} V. B\"ogelein, F. Duzaar,  P. Marcellini, {\textit {Existence of evolutionary variational solutions
via the calculus of variations}}, J. Differ. Equ., 256(12) (2014), 3912--3942.

\bibitem{BDMS1} V. B\"ogelein, F. Duzaar, P. Marcellini, and S. Signoriello, {\textit {Nonlocal diffusion equations}}, J.
Math. Anal. Appl., 432(1) (2015), 398--428.

\bibitem{BDMS2} V. B\"ogelein, F. Duzaar, P. Marcellini, and S. Signoriello, {\textit {Parabolic equations and the
bounded slope condition}}, Ann. Inst. H. Poincar\'e C Anal. Non Lin\`eaire, 34(2) (2017), 355--379. 





\bibitem{BSV} M. Bonforte, Y. Sire, J. L. V\'azquez, {\textit{
Optimal existence and uniqueness theory for the fractional heat equation}},
Nonlinear Analysis: Theory, Methods $\&$ Applications 153
(2017),
 142--168.
 
 \bibitem{BP} C. Br\"andle, A. De Pablo,  {\textit{Nonlocal heat equations: Regularizing effect, decay estimates and Nash inequalities}}, Commun. Pure Appl. Anal. 17(3) (2018), 1161--1178.



\bibitem{CCR} E. Chasseigne, M. Chaves, J. D. Rossi, {\textit{Asymptotic behavior for nonlocal diffusion equations}}, J. Math. Pures Appl. (9) 86 (2006), 271--291.







\bibitem{DG}{E. De Giorgi}, \textit{Conjectures concerning some evolution problems. A celebration of John F. Nash, jr}. Duke Math. J. 81 (1996), 61-100.


 
 \bibitem{DG2}{E. De Giorgi}, {Selected Papers}. 
Springer-Verlag, New York, 2006, (edited by L. Ambrosio, G. Dal Maso, M. Forti, M. Miranda, and S. Spagnolo).




\bibitem{F}P. Fife, {\textit{Some nonclassical trends in parabolic and parabolic-like evolutions}}, in: Trends in Nonlinear Analysis, Springer-Verlag, Berlin, 2003,
pp. 153--191.

 \bibitem{LS} M. Liero, U. Stefanelli, {\textit{A new minimum principle for Lagrangian mechanics}}, J. Nonlinear Sci. 23 (2) (2013), 179--204.
 \bibitem{LS2}  M. Liero, U. Stefanelli, \textit{Weighted inertia-dissipation-energy functionals for semilinear equations}, Boll. Unione Mat. Ital. (9) 6  (2013), no. 1, 1--27.
\bibitem{MP} E. Mainini, D. Percivale, \textit{Newton’s second law as limit of variational problems},  Adv. Cont. Discr. Mod.  2023, 20 (2023).

\bibitem{MP2} E. Mainini, D. Percivale,
{\textit {On the weighted inertia-energy approach to forced wave equations}},
J. Differ. Equ. 385 (2024),
 121--154.
 
 \bibitem{M} P. Marcellini. {\textit A variational approach to parabolic equations under general and $p, q$-growth
conditions}, Nonlinear Anal. 194:111456, 17, (2020).

\bibitem{PT} H. Prasad, V. Tewary, {\textit Existence of variational solutions to nonlocal evolution equations via convex minimization}, ESAIM Control Optim. Calc. Var. 29  (2023) 2.

\bibitem {ST}{E. Serra, P. Tilli}, {\textit { Nonlinear wave equation as limits of convex minimization problems: proof of a conjecture by De Giorgi}, }Annals of Math.  175 (2012),  1551--1574.
\bibitem{ST2}{E. Serra, P. Tilli}, {\textit {A minimization approach to hyperbolic Cauchy problems}}. J. Eur. Math. Soc. 18 (2016), no. 9, 2019--2044.


\bibitem{S} {U. Stefanelli}, \textit{The De Giorgi conjecture on elliptic regularization}, Math. Models Methods Appl. Sci. 21 (2011), 1377--1394.

\bibitem{S2} U. Stefanelli, {\textit{The weighted Inertia-Energy-Dissipation principle}},  Mathematical Models and Methods in Applied Sciences 35, No. 02 (2025), 223--282. 







\bibitem{TT1} L. Tentarelli, P. Tilli, {\textit{De Giorgi’s approach to hyperbolic Cauchy problems: the case of nonhomogeneous equations}}, Comm. Partial Differential Equations 43 (4) (2018), 677--698.

\bibitem{TT2} L. Tentarelli, P. Tilli, {\textit{An existence result for dissipative nonhomogeneous hyperbolic equations via a minimization approach}},
J. Differential Equations 266 (2019), 5185--5208.







\bibitem{V} J. L. V\'azquez, {\textit{The Mathematical Theories of Diffusion: Nonlinear and Fractional Diffusion}}, In: Bonforte, M., Grillo, G. (eds) Nonlocal and Nonlinear Diffusions and Interactions: New Methods and Directions. Lecture Notes in Mathematics, vol 2186. Springer, Cham (2017).

\bibitem{V2} 
J. L. V\'azquez, J. L. {\textit {Asymptotic behaviour for the fractional heat equation in the Euclidean space}}. Complex Variables and Elliptic Equations 63(7–8) (2017), 1216--1231. 



\end{thebibliography}
\end{document}